\documentclass{article}
\usepackage{hyperref}

\usepackage[dvipsnames]{xcolor}
\usepackage{hyperref}
\usepackage{cleveref}
\colorlet{mylinkcolor}{RoyalBlue}
\colorlet{mycitecolor}{magenta}
\colorlet{myurlcolor}{Aquamarine}

\hypersetup{
	linkcolor  = mylinkcolor,
	citecolor  = mycitecolor,
	urlcolor   = myurlcolor,
	colorlinks = true,
}
\usepackage{amsmath, amsfonts, amssymb, mathrsfs,amsthm}
\usepackage{caption, enumerate, enumitem}
\usepackage{graphicx}
\usepackage[margin=1in]{geometry}
\usepackage[numbered,framed]{matlab-prettifier}
\usepackage{threeparttable,booktabs}

\newtheorem{thm}{Theorem}[section]
\newtheorem{corollary}{Corollary}[thm]
\newtheorem{lemma}[thm]{Lemma}
\theoremstyle{definition}
\newtheorem{definition}{Definition}[section]

\theoremstyle{remark}
\newtheorem{remark}{Remark}
\theoremstyle{remark}
\newtheorem{example}{Example}

\newtheorem{alg}{Algorithm}

\usepackage[title]{appendix}

\newcommand{\RR}{{\mathbb R}}  
\newcommand{\ZZ}{{\mathbb Z}}   
 
\newcommand{\mc}[1]{{\mathcal{#1}}} 
\newcommand{\labeleq}[1]{\label{eq:#1}}
\newcommand{\refeq}[1]{\text{(\ref{eq:#1})}}
\newcommand{\ip}[2]{\langle#1,#2\rangle}
\renewcommand{\(}{\left(}
\renewcommand{\)}{\right)}

\usepackage{authblk}

\author[1]{ Christina Frederick\thanks{christin@njit.edu}}
\author[2]{Kasso A.~Okoudjou\thanks{kasso@mit.edu}}
\affil[1]{Department of Mathematical Sciences, New Jersey Institute of Technology}
\affil[2]{Department of Mathematics, Massachusetts Institute of Technology}

\begin{document}

\title{Finding duality for Riesz bases of exponentials on multi-tiles}

\maketitle
\begin{abstract}

It is known \cite{Cabrelli2018, Grepstad2014, Kolountzakis2013} that if $\Omega \subset \mathbb{R}^{d}$ belongs to a class of multi-tiling domains when translated by a lattice $\Lambda$, there exists a Riesz basis of exponentials for $L^{2}(\Omega)$ constructed using $k$ translates of the dual lattice $\Lambda^*$. In this paper, we give an explicit construction of the corresponding biorthogonal dual Riesz basis. We also extend the iterative reconstruction algorithm introduced in \cite{Frederick2018} to this setting.
\end{abstract}


\section{Introduction}

This paper centers on Riesz bases of exponentials $\{e_l(x):=e^{2\pi i l \cdot x}\}_{l\in L}$
for the space $L^2(\Omega)$, where  $\Omega \subset \RR^d$ is a set of positive and finite Lebesgue measure, and $L \subset \RR^d$ is a countable set. The set  $\{e_l(x)\}_{l\in L}\subset L^2(\Omega)$
is a	Riesz basis for $L^2(\Omega)$ if each $f\in L^2(\Omega)$ has the unique representation
\begin{equation}\label{def-rb1}
f(x) = \sum_{l\in L} c_l e_l(x),
\end{equation} 
where 
the coefficients 
$\{c_l\}_{l\in L} \in \ell^2(L)$ satisfy
\begin{equation}\labeleq{def-rb2}
A\|f\|_{L^2(\Omega)}^2\leq \sum_{l\in L}|c_l|^2\leq B \|f\|_{L^2(\Omega)}^2
\end{equation} 
for some constants $0<A\leq B < \infty$. In this case, there exists a (unique) dual Riesz basis $\{g_l(x)\}_{l\in L}\subset L^2(\Omega)$ that satisfies the biorthogonality condition
\begin{align}
\ip{e_l}{g_{l'}}=\begin{cases} |\Omega| &\text{ if }l=l'\\  0& \text{ otherwise},\end{cases}\labeleq{dualdef}
\end{align}
and the coefficients in (\ref{def-rb1}) are given by $c_l=\frac{1}{|\Omega|}\ip{f}{g_l}$.

Although no general proof of the statement exists, there are many cases where it is known that a set $\Omega$ admits a Riesz basis of exponential functions .  These cases include when $\Omega$ is a finite union of co-measurable intervals in $\RR$ or multi-rectangles in $\mathbb{R}^{d}$  \cite{Kozma2014, DeCarli2015} and when $\Lambda$ is a stable set of sampling for the Paley-Wiener space $PW_\Omega$  \cite{Bezuglaya1993, Jaffard1991, Lyubarskii1997, Seip1995}. It was recently established in \cite{Debernardi2019} that any convex polytope that is centrally symmetric and whose faces of all dimensions are also centrally symmetric admits a Riesz basis of exponentials. For more details on properties of families of exponentials, we refer to \cite{Young01}.

However, to the best knowledge of the authors, no explicit algorithms or formulas for the corresponding dual Riesz bases are available in the literature. As seen from~\eqref{def-rb1}, knowing the biorthogonal dual is important for the reconstruction of any function in $L^2(\Omega)$. One of the goals of this paper is to construct biorthogonal Riesz bases for $L^2(\Omega)$ for a class of domains $\Omega\subset \mathbb{R}^d$ of finite, positive Lebesgue  measure.

The interest in Riesz bases of exponentials stems partially from the Fuglede Conjecture \cite{Fuglede1974}, which asserts that the set of exponentials $\{e_l(x)\}_{l\in L}$ is an orthogonal basis for $L^2(\Omega)$ if and only if $\Omega$ tiles $\RR^d$ with respect to the discrete set $L\subset \RR^d$. In this case, the system of exponentials is self-dual, that is $e_l\equiv g_l$ for $l\in L$ and $A=B=\frac{1}{|\Omega|}$ in~\refeq{def-rb2}. The Fuglede conjecture has been disproved in both directions when $d\geq 3$ \cite{Fuglede1974, Matolcsi2005, Kolountzakis2006, Kolountzakis2006b} but remains open when $d=1,2$. It has recently been proved in \cite{Lev2019} that the Fuglede conjecture does hold for convex domains in all dimensions. Removing the rigidity imposed by orthonormality leads naturally to the problem of obtaining  Riesz bases of exponentials.

Let $\Lambda$ be a full lattice in $\RR^d$ and $k$ be a natural integer. We say that a Lebesgue measurable set $\Omega \subset \mathbb{R}^d$ of finite positive measure  is a $k$-tile for  $\Lambda$ (or multi-tiling subset of $\mathbb{R}^d$) if 
\begin{align*}
\sum_{l\in \Lambda}\chi_{\Omega}(x - l) = k \text{ for almost all }x\in \mathbb{R}^{d}.\end{align*}

When $\Omega\subset \RR^d$ is an {\it admissible} $k$-tile for a full lattice $\Lambda\subset \RR^d$ (defined in Remark \ref{rem:dual}),  then it admits a Riesz basis of exponentials  \cite{Cabrelli2018, Grepstad2014, Kolountzakis2013}. More specifically,  there exists a set of vectors $\{a_{s}\}_{s=1}^{k}\subset \mathbb{R}^d$  such that the exponentials 
\begin{align} \{ e_l(x)=e^{2\pi i l \cdot x}\}_{l\in L},\qquad L = \bigcup_{s=1}^k \Lambda^*+a_s \labeleq{rb}
\end{align}
form a Riesz basis for $L^2(\Omega)$,  where $\Lambda^*$ is the dual lattice of $\Lambda$.

The first goal of the paper, accomplished in Section \ref{sec:dual}, is to introduce a procedure for constructing the biorthogonal dual Riesz basis for any multi-tiling domain that is known to admit a Riesz basis of exponentials of the form \refeq{rb}. The second goal of this paper is to derive an iterative (and adaptive) algorithm for performing the pointwise reconstruction of functions $f\in L^2(\Omega)$ given {\it data} of the form
$\{\langle f,e_{l} \rangle\}_{l\in L}$. The algorithm is deterministic, it only involves inverting  1D Vandermonde systems, and it establishes a new framework for finding a set of vectors $\{a_s\}_{s=1}^k\subset \RR^d$ for which \refeq{rb} is guaranteed to be a Riesz basis with Riesz bounds that are quantifiable using analytical formulas for 1D Vandermonde matrices. 

As we will show in Section \ref{sec:sampalg}, this algorithm extends the results that recently appeared in \cite{Frederick2018} to more general multi-tiling sets. Compared to the iterative procedure in \cite{Behmard2002, Behmard2006}, our algorithm solves invertible 1D Vandermonde systems at each iteration. A related algorithm in \cite{Faridani1994} relies on the existence of solutions to a linear system. In contrast, the algorithm presented here chooses the lattice shifts to ensure invertibility. The algorithm introduced here more closely compares to the construction in \cite{Lyubarskii1997}, however instead of creating a $p$-th order sampling procedure at each iteration, a first-order sampling is used at each iteration.  Another line of investigation that compares to the present work is considered in \cite{Grochenig2001, Kammerer2019} where numerical sampling algorithms for multivariate trigonometric polynomials are used to derive approximations of infinite-dimensional bandlimited functions from nonuniform sampling.  Recent results in \cite{Kammerer2019} employ numerical methods for sampling along random rank-1 lattices  for  a  given  frequency  set  to  approximate  multivariate  periodic  functions,  focusing  on  the  error  in approximation.  In contrast, we formulate a deterministic algorithm and provide guarantees for exact reconstruction in cases where the domain $\Omega$ is fully known.

\section{Finding duals}\label{sec:dual}

The main result in this section offers an explicit construction of biorthogonal systems of exponentials corresponding to a class of multi-tiling sets. Throughout the paper, we let $\Lambda = M\ZZ^d$ be a full lattice generated by the basis vectors $M = [m_1, \hdots, m_n]\in \RR^{d\times d}$. The canonical dual lattice is $\Lambda^{*} = \{M^{-T}z, z\in \mathbb{Z}^d\}$ and $\Pi_{\Lambda}$ denotes the fundamental domain $\Pi_{\Lambda} =M \mathbb{T}^d$. 

Let $\Omega\subset\mathbb{R}^d$ be a measurable domain with $0<|\Omega|<\infty$ that is a $k$-tile for $\Lambda$. Then, there exists a partition $\Omega=\Omega_1\cup\Omega_2\cup\hdots \cup \Omega_k\cup E$, where $E$ is a set of measure zero, and $\Omega_j$ are mutually disjoint measurable sets that are each a fundamental domain of $\Lambda$ (Lemma 1 in \cite{Kolountzakis2013}). Let $\Lambda_x(\Omega):=\Lambda_x = \{\lambda \in \Lambda \mid x +\lambda \in \Omega\}$.
Denote the cardinality of the finite set $\mc{S}$ by $\#\mc{S}$. Then $\# \Lambda_{x} = k$ for almost every $x\in \Pi_{\Lambda}$ and we define the points $\{\lambda_r(x)\} \subset\Lambda$ to be the unique lattice points that satisfy
\begin{align}
x+\lambda_r(x)\in \Omega_r, \qquad 1\leq r\leq k
\labeleq{Lambdax}.
\end{align}
The mapping $\omega_r: \Pi_{\Lambda} \rightarrow \Omega_r$ given by $x \rightarrow x+\lambda_r(x)$ is then invertible. 

Our first main result is the following.

\begin{thm}\label{thm:dualbasis} Let $\Omega\subset\mathbb{R}^d$ be a $k$-tile for a full lattice  $\Lambda$, and define 	$\{\lambda_r(x)\}_{r=1}^k$ and $\{\omega_r(x)\}_{r=1}^k$ by \refeq{Lambdax}. If there exists a set of vectors $\{a_{s}\}_{s=1}^{k}\subset\mathbb{R}^d$ and positive constants $\alpha$ and $\beta$ such that the matrix function $V=V(x)\in \mathbb{C}^{k\times k}$ with entries  \begin{align}  (V(x))_{sr}=e^{-2\pi i \lambda_r(x) \cdot a_{s}}, \qquad 1\leq s,r\leq k, \labeleq{V}\end{align}
has uniformly bounded singular values, $0<\alpha<\sigma_{i}\(V(x)\)<\beta <\infty$, $1\leq i\leq d$, then the following two families of functions form a pair of  biorthogonal Riesz bases  for $L^2(\Omega)$:
\begin{align}\labeleq{bio-basis}
\{ e_l(x)\}_{l\in L},\qquad  \{ g_l(x)\}_{l\in L},\qquad L = \bigcup_{s=1}^k \Lambda^*+a_s 
 \end{align}
where, for $\lambda^*\in\Lambda^*$ and $1\leq s\leq k$,
\begin{align}
g_{\lambda^{*}+a_{s}}(x)=e_{\lambda^{*}+a_{s}}(x) \( k\sum_{r=1}^{k} {(V(\omega_r^{-1}(x)))_{sr}}(V(\omega_r^{-1}(x))^{-1})_{rs}\chi_{\Omega_r}(x)\).	\labeleq{hlambda}
\end{align}
Here $\chi_{\Omega_r}$ denotes the indicator function of $\Omega_r$.
\end{thm}

\begin{proof}

Let $f\in L^2(\Omega)$. Then the mappings from $\ell^2\rightarrow L^2(\Pi_\Lambda)$ given by \[\{\langle f,e_l \rangle\}_{l\in \Lambda^*+a_s}\rightarrow \sum_{l\in \Lambda^*+a_s} \langle f, {e}_{l}\rangle  {e}_{l}(x)\]
are well-defined (in fact, isometries) for all $1\leq s\leq k$  by the orthogonality of the exponentials $\{e_l\}_{l\in \Lambda^*+a_s}$ in the space $L^2(\Pi_\Lambda)$. It follows from the Poisson summation formula that the sequences $\{\langle f,e_l \rangle\}_{l\in \Lambda^*+a_s}$ and $\{f(\omega_r(x))\}_{r=1}^k$ are related by the following set of linear equations for almost every $x \in \Pi_{\Lambda}$: 
\begin{align}
\sum_{l\in \Lambda^*+a_s} \langle f, {e}_{l}\rangle  {e}_{l}(x)= 
\text{vol}(\Lambda)\sum_{r=1}^k  (V(x))_{sr}{f}(\omega_r(x)) \labeleq{veq}, \qquad 1\leq s\leq k.
\end{align}
By assumption, the lattice shifts $\{a_s\}$ are chosen so that $V(x)$ is invertible, leading to the following $\ell_{2}$ estimate
\begin{align*}
\|V(x)^{-1}\|^{-2} \sum_{r=1}^{k}|f(\omega_r(x))|^{2}	 \leq  \sum_{s=1}^{k} |\sum_{l\in \Lambda^*+a_s} \langle f, {e}_{l}\rangle  {e}_{l}(x)|^{2} \leq \|V(x)\|^2\sum_{r=1}^{k}|f(\omega_r(x))|^{2}.
\end{align*}
Since both $\|V(x)\|=\max_{1\leq i\leq d} \sigma_{i}(V(x))<\beta$ and $\|V(x)^{-1}\|=1/\min_{1\leq i\leq d} \sigma_{i}(V(x))>1/\alpha$ are uniformly bounded, integrating over $\Pi_{\Lambda}$ produces

\begin{align}
A\|{f}\|^{2}_{L^{2}(\Omega)}	 \leq \sum_{l\in L}| \langle {f}, e_{l} \rangle|^{2} \leq   B \|f\|^{2}_{L^{2}(\Omega)},	\qquad L=\bigcup_{s=1}^k\Lambda^*+a_s,\labeleq{felbounds}
\end{align}
where $0<A\leq B<\infty$ are constants. 

To show (\ref{def-rb1}), let  $l=\eta^*+a_{s}$ for a fixed lattice point $\eta^*\in\Lambda^{*}$ and $1\leq s\leq k$ and define the functions $G_{l;r}(x), 1\leq r\leq$ to be the unique solutions to the system
\begin{align}
\begin{pmatrix}0 \\ \vdots\\  {e}_{l}(x) \\\vdots\\
0 \end{pmatrix} &=\frac{1}{k} V(x)\begin{pmatrix} G_{l;1}(x)\\ \vdots\\ 
G_{l;r}(x)\\ \vdots\\ G_{l;k}(x)) \labeleq{sys1} \end{pmatrix}.
\end{align}
Since $\|(V(x))^{-1}\|$ is also uniformly bounded, $G_{l;r}$ are functions in $L^2(\Pi_\Lambda)$. Then the function uniquely defined by $g_{l}(\omega_r(x)) = G_{l;r}(x)$ is in  $L^2(\Omega)$ and solves \refeq{veq} for almost every $x\in \Pi_\Lambda$:

\begin{align}\frac{1}{|\Omega|} \begin{pmatrix}
{e_{a_{1}}(x)} & & & & 0\\
&\ddots & & & \\
& &{e_{a_{s}}(x)} & & \\
& & &\ddots & \\
0& & & &{e_{a_{k}}(x)} \\
\end{pmatrix}
\begin{pmatrix} \sum_{\lambda^* \in \Lambda^*} \langle g_{l}, {e}_{\lambda^*+a_1}\rangle  {e}_{\lambda^*}(x) \\ \vdots\\  \sum_{\lambda^* \in \Lambda^*} \langle g_{l}, {e}_{\lambda^*+a_{s}}\rangle  {e}_{\lambda^*}(x) \\\vdots\\
\sum_{\lambda^* \in \Lambda^*} \langle g_{l}, {e}_{\lambda^*+a_k}\rangle  {e}_{\lambda^*}(x) \end{pmatrix} &=\frac{1}{k} V(x)\begin{pmatrix} g_{l}(\omega_1(x))\\ \vdots\\ 
g_{l}(\omega_r(x))\\ \vdots\\ g_{l}(\omega_k(x))  \end{pmatrix}.\labeleq{sys2}
\end{align}
Here we used that $\text{vol}(\Lambda)=\frac{|\Omega|}{k}$ because $\Omega$ is a $k$-tile for $\Lambda$. Since the right hand side of the equations \refeq{sys1} and \refeq{sys2} are equal, we can  equate the left hand side of both equations to obtain

\begin{align*}
\frac{1}{|\Omega|}\sum_{\lambda^* \in \Lambda^*}\langle g_{l}, {e}_{\lambda^*+a_{s'}}\rangle  {e}_{\lambda^*}(x)= \begin{cases} {e}_{\eta^*}(x), & s=s'\\
0, & s\neq s'\end{cases}.
\end{align*}  
 Since $\{e_{\lambda^*}\}_{\lambda^*\in \Lambda^*}$ form an orthogonal basis for $L^2(\Pi_\Lambda)$, this implies that $\langle g_{\eta^*+a_{s}}, {e}_{\lambda^*+a_{s'}}\rangle=0$ for $\lambda^*+a_{s'}\neq \eta^*+a_{s}$ and $|\Omega|$ otherwise. Therefore there is a unique sequence $\{g_{l}\}_{l\in L}\subset L^2(\Omega)$ satisfying \refeq{dualdef} and subsequently all functions $f\in L^2(\Omega)$ have the unique representation (\ref{def-rb1}), with $c_l=\langle f, g_l\rangle$. Then $\sum_{l\in L}|\langle f, g_l\rangle|^2= \sum_{l\in L}|\langle f, \sum_{l'\in l}\langle g_l, e_l'\rangle e_l'\rangle|^2 = \sum_{l\in L}|\langle f, e_l\rangle|^2$. By \refeq{felbounds}, the family $\{e_l\}_{l\in L}\subset L^2(\Omega)$ satisfies the frame condition \refeq{def-rb2}. We have shown that $\{e_l\}_{l\in L}$ is a Riesz basis of exponentials. 
 
 To find an explicit formula for the dual Riesz basis $\{g_l\}_{l\in L}$, the system \refeq{sys1} can be explicitly solved using only the ${s}^\text{th}$ column of the matrix $(V(x))^{-1}$:
\begin{align*}
k (V(x)^{-1})_{rs}{e}_{l}(x)& ={g_{l}}(\omega_r (x)), \quad 1\leq r\leq k, \text{ a.e. } x \in \Pi_{\Lambda}.
\end{align*}
Since ${e}_{l}(\omega_r(x) )=e^{2\pi i l\cdot (\lambda_r(x) +x)}=e_{l}(x) e^{2\pi i l\cdot \lambda_r(x)} =e_{l}(x) \overline{(V(x))_{sr}}$, we can write this as \begin{align}
k (V(x))_{sr} (V(x)^{-1})_{rs} {e}_{l}(\omega_r(x))& ={g_{l}}(\omega_r(x) ), \quad 1\leq r\leq k, \text{ a.e. } x \in \Pi_{\Lambda} \labeleq{geta}.
\end{align}
The functions in \refeq{hlambda} are obtained by extending the domain in \refeq{geta} to all $x\in \Omega$ using the inverse functions $\omega_r^{-1}(x)=x-\lambda_r(x)$ for $x\in \Omega_r$.

\end{proof}

\begin{remark}\label{rem:dual}
\begin{enumerate}
\item This result holds for all multi-tiling sets for which there exists a Riesz basis of exponentials of the form \refeq{rb}, including \textit{admissible} multi-tiling sets. A multi-tiling set for a full lattice $\Lambda$ is admissible if there exists an element of the dual lattice $v\in \Lambda^*$ and an integer $n\in\ZZ_+$ such that $v\cdot \lambda_1(x), \hdots, v\cdot \lambda_k(x)$ are distinct integers \textit{modulo} $n$ for almost every $x$ \cite{Grepstad2014, Kolountzakis2013, Cabrelli2018, Cabrelli2020}.
\item For $k=1$ the theorem implies that $g_l(x)=e_l(x)\chi_{\Omega}(x)=e_l(x)$ and the system is an orthogonal basis.
\item If the system is self-dual, that is $g_l=e_l$, then \refeq{geta} implies that  $(V(x)^{-1})_{rs}=\frac{1}{k} e^{2\pi i \lambda_r(x)\cdot a_{s}}$. Therefore, Theorem \ref{thm:dualbasis} shows that the system \refeq{bio-basis} forms an orthogonal basis if and only if  $V(x)^*V(x)=kI$ meaning that $V(x)$ is a (log) Hadamard matrix \cite{KolMat06}.

\end{enumerate}
\end{remark}

\section{Finding Riesz bases of exponentials}\label{sec:sampalg}

Suppose  $\Omega\subset\RR^d$ is a $k$-tile for the full lattice $\Lambda=M\ZZ^d$. Theorem \ref{thm:dualbasis} provides sufficient conditions on the set of {\it lattice shifts} $\{{a}_{s}\}_{s=1}^{k}\subset{\RR^{d}}$ so that the two families of functions defined in \refeq{bio-basis} form a pair of biorthogonal Riesz bases for $L^2(\Omega)$. These conditions are based on the uniform bound on the singular values of the $k\times k$ matrices $V(x)$ defined in \refeq{V}. For admissible domains $\Omega,$ the lattice shifts can be chosen so that the matrices $V(x)$ are square Vandermonde matrices $(V(x))_{rs}=w_{r}^{s-1}$ in which the nodes $w_{r}=e^{2\pi i (\lambda_r(x)\cdot v)/n}$ are a subset of the $n^\text{th}$ roots of unity $\{w_{r}\}_{r=1}^{k}\subset \{e^{2\pi i j/n}\}_{j=1}^{n}$ for almost every $x$. Since the distance between any two nodes has a uniform lower bound, that is, $\underset{r\neq r'}{\max}|w_r - w_{r'}|\geq |1-e^{2\pi i/n}|>0$, it follows from \cite{Gautschi1990} that the matrices $V(x)$ have uniformly bounded singular values for almost every $x\in \Pi_\Lambda$.

In general, determining the invertibility of Fourier matrices is challenging, especially when $d$ is large. Even when invertibility is guaranteed, directly solving the linear system \refeq{veq} becomes increasingly difficult as $k$ and $d$ grow. In this section,  we will discuss a strategy for reconstructing functions using a family of exponentials and finding a Riesz basis of exponentials for a class of multi-tiles.

To keep the notations simple, we collect in Table \ref{table:los} a list of symbols used in this section.

For $x\in \Pi_\Lambda$ define the {\it frequency set} $\mc{M}^d = \mc{M}^d (x) = \Lambda_x $. Since the {\it shift index set} $\mc{K}^d(\mc{M}^d)$ is defined recursively, we first prove the following lemma.

\begin{table}\caption{List of Symbols used in $\S$\ref{sec:sampalg}.}\label{table:los}
\centering
\begin{threeparttable}[t]
\renewcommand{\arraystretch}{1.5}
\begin{tabular}{rl }
\toprule
{$\mc{M}^d$}&{A given set of $k$ vectors $m_1,\hdots m_k$ in $\RR^d$}\\[.5em]
{$\mathcal{M}^{l}$}&$=\{m_i\in \RR^l \mid (m_i,m')\in \mathcal{M}^d \text{ for some } m'\in \RR^{d-l}  \}$\tnote{1}\\[.5em]
{$\mathcal{Z}^{l}_i$}&$=\begin{cases} \mc{M}^1 & l=i=1 \\ \{z_{l}\in \RR \mid (m_i, z_{l}) \in \mathcal{M}^{l}\}& l\geq 2,\end{cases}$\\[1.5em]

$\mc{M}^{l-1}_i$&{$=\{m'_{j} \in \mc{M}^l \mid j\geq i\}$\qquad $l\geq 2$}\\[.5em]
{$\mathcal{Q}_{i}(\mathcal{M}^{l})$}&{$=\begin{cases}\{ 0, \hdots, \#\mathcal{Z}^{l}_{i}- 1 \}& i=1\\
	 \{ \#\mathcal{Z}^{l}_{i-1}, \hdots, \#\mathcal{Z}^{l}_{i}- 1 \}& i\geq 2 \end{cases}$ } \\[1.5em]
	 
$\mc{K}^{l}(\mathcal{M}^{l})$&$=\begin{cases}
	\{0,\hdots, \#\mathcal{M}^{1}-1\}, & l=1\\  \bigcup_{i=1}^{\#\mathcal{M}^{l-1}}  \mc{K}^{l-1}(\mc{M}^{l-1}_i) \times \mathcal{Q}_{i}(\mathcal{M}^{l})& l\geq 2\end{cases}$\\[1.5em]
\bottomrule
\end{tabular}

\begin{tablenotes}
\item[1] {\footnotesize $\mathcal{M}^{l-1}=\{m_{i}\}$ is enumerated so that $\#\mathcal{Z}^{l}_{i} \leq\#\mathcal{Z}^{l}_{j}$ for $i\leq j$.}
\end{tablenotes}

\end{threeparttable}%

\end{table}

\begin{lemma} For each $l=1, \hdots d$, 
$\# \mc{K}^{l}(\mathcal{M}^{l}) = \# \mathcal{M}^{l}$.\label{lem:shift}
\end{lemma}

\begin{proof} By induction. This is true for $l=1$ by definition.
Consider $1<l\leq d$. The induction assumption asserts that there are $\# \mathcal{M}^{l'}$ elements in $\mc{K}^{l'}(\mathcal{M}^{l'})$ for $1\leq l'<l$. Therefore there are $\# \mathcal{M}^{l-1}$ elements in $\mc{K}^{l-1}(\mc{M}^{l-1})$ and for $i>1$, $\# \mathcal{M}^{l-1}  - i+1 = \# \mathcal{M}^{l-1}_i=\#\mc{K}^{l-1}(\mathcal{M}^{l-1}_i)$. Then, there are $\#\mathcal{Z}^{l}_{1}$ in $\mc{Q}_{1}(\mathcal{M}^{l})$ and for $i>1$, $(\#\mathcal{Z}^{l}_{i}-\#\mathcal{Z}^{l}_{i-1})$ elements in $\mc{Q}_{i}^{l}(\mathcal{M}^{l})$.
Summing this up,
\begin{align*}\#\mc{K}^{l}(\mc{M}^l) & = \#\mc{M}^{l-1} \#\mc{Z}^{l}_1 + \sum_{i=2}^{\#\mc{M}^{l-1}}( \#\mc{M}^{l-1} - i+1)(\#\mc{Z}^{l}_{i}-\#\mc{Z}^{l}_{i-1})\\& =\sum_{i=1}^{\#\mc{M}^{l-1}} \#\mc{Z}^{l}_{i} = \#\bigcup_{i=1}^{\# \mc{M}^l}\{(m_i, z_l)\mid z_l\in\mc{Z}^l_i\}=\# \mathcal{M}^{l}.
\end{align*}

\end{proof}

The construction of the shift index set $\mc{K}^d(\mc{M}^d)$ is based on a tree structure admitted by the frequency  set $\mc{M}^d$.
The sets $\mc{M}^l$ correspond to the collection of vectors in each parent node in the $l^{th}$ level of the tree. The ordering $\mc{M}^{l-1}=\{m_i\}$ corresponds to nodes in increasing order of the number of immediate children. The sets $\mc{Z}^l_p$ correspond to the last coordinate of the children of the $p^{th}$ vector in $\mc{M}^{l-1}$.

\begin{example} \label{ex:z4}
Let $k=10$ and consider the frequency set
\begin{align*}
\mathcal{\mathcal{M}}^4 = \left\{
\begin{tabular}{cccc}
(1,1,1,1), &(2,1,1,1), & (3,1,1,1), & (4,1,1,1),\\
&(2,2,1,1), &(3,2,1,1), & (4,2,1,1),\\
&(2,2,1,2), &  (3,2,2,1), & (4,3,1,1)
\end{tabular}\right\}.
\end{align*}
The tree diagram produced by $\mc{M}^4$ is illustrated in Figure \ref{fig:ex4}.

\begin{figure}[ht]
\includegraphics[width=\linewidth]{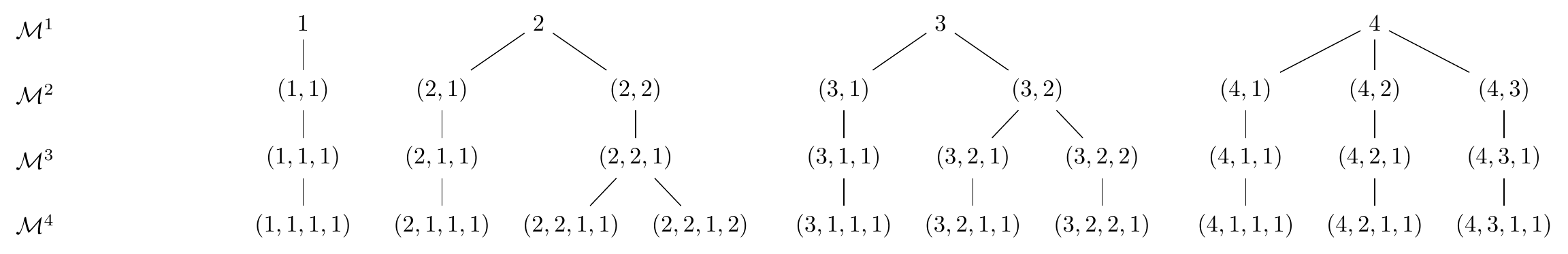}
\caption{Tree structure produced by $\mc{M}^4$ in Example \ref{ex:z4}.}\label{fig:ex4}\end{figure}

\end{example}

\subsection{Finding Riesz bases of exponentials}

We first define a weaker notion of admissibility.

\begin{definition}\label{def:weakadmi}
Let $\Omega\subset\RR^d$ be a $k$-tile for a full lattice $\Lambda$. 
For $v\in\mathbb{R}^{d}$ and $q\in \ZZ_{+}^{d}$ we say that $\Omega$ is {\it weakly} $   (v,q)${\it-admissible} if for almost every $x\in \Pi_{\Lambda}$ and $\mc{M}^{d}=\mc{M}^{d}(x) = \Lambda_{x}$, the following condition is satisfied:
\[\#\{v_{l} z_{l}\mod q_l \mid z_{l}\in \mc{Z}^{l}_p\} =  \#\mc{Z}^l_p,\quad  1\leq p \leq N_{l-1},\quad  1\leq l\leq d,\]
where $N_0=1$ and $N_l=\#\mc{M}^{l}$ if $1\leq l \leq d$.
\end{definition}
These conditions state that the numbers $v_l z_l$ are distinct modulo $q_l$. In particular, any bounded $k$-tile for a lattice $\Lambda$ is weakly admissible for some pair $(v,q)$. Weakly admissible domains are not necessarily admissible as defined in Remark \ref{rem:dual}.

\begin{thm}\label{thm:vander}
Suppose that $\Omega\subset \RR^d$ is a  $k$-tile for a full lattice $\Lambda$ that is weakly $(v,q)$-admissible for some $q\in \ZZ_+^{d}$ and  $v\in \RR^{d}$. Then, for almost every $x\in\Omega$, $f(x)$ can be uniquely determined from the data
\begin{align*}
\{ \langle f, e_{l}\rangle\}_{l\in L},\qquad  L= \bigcup_{s=1}^k \Lambda^*+a_s,
\end{align*}
where the lattice shifts are given by
\begin{align}
\{ a_s(x) = \delta j_s + \eta \mid  j_s\in \mc{K}^d(\mc{M}^d(x))\}\labeleq{shiftset}, \qquad
\end{align}
for any diagonal matrix $\delta\in \RR^{d\times d}$ of the form $\delta_{ll}=v_{l}/{q_{l}}$, $1\leq l \leq d$ and $\eta$ is any point in the dual lattice $\Lambda^*$.

\end{thm}

\begin{proof} 
Consider the system of equations \refeq{veq} given by $(F_{j_s}(x))_{s=1}^k=V(x) (F^{\lambda_r}(x))_{r=1}^k$ for almost every $x \in \Pi_{\Lambda}$, where $V(x)=V$ has the form \refeq{V} and
\begin{align}
F_{j_s}(x)&=F_{j_s}=\frac{1}{\text{vol}(\Lambda)}\sum_{\lambda^{*}\in \Lambda^{*}} \langle f, {e}_{\lambda^{*}+a_{s}}\rangle{e}_{\lambda^{*}+a_{s}}(x),\labeleq{Faj}\\ 
{F}^{\lambda_{r}}(x)
&={F}^{\lambda_r}=  {f}\left(\omega_r(x)\right)\labeleq{Fr}.
\end{align}
Notice that $e^{2\pi i \lambda_r(x)\cdot(\delta j_s + \eta)} = e^{2\pi i \lambda_r(x)\cdot \delta j_s}$ so it suffices to consider shifts of the form $a_s=\delta j_s$.
Since for almost every $y\in \Omega$, $y=\omega_r(x)$ for some $x\in \Pi_\Lambda$ and $1\leq r\leq k$, we will prove the unique recovery of $f(x)$ by showing the invertiblity of the matrices $V^{l}(x)=V^l$, $1\leq l\leq d$, defined by
\begin{align}
(V^{l})_{sp}=
e^{-2\pi i \delta'' j''_{s}\cdot m_{p}},\qquad j''_s\in \mc{K}^l(\mc{M}^l), \quad m_p\in \mc{M}^l. \labeleq{Vlinvertible}
\end{align}
Here, $\mc{K}^{l}(\mc{M}^{l})$ is the set of $\#\mc{M}^l$ shift indices constructed in Lemma \ref{lem:shift}, and we define the diagonal matrices $\delta'=(\delta_{ij})_{ l+1\leq i, j\leq d}$ and  $\delta''=(\delta_{ij})_{ 1\leq i, j\leq l}$. This is equivalent to proving the uniqueness of solutions of the linear systems
\begin{align}
\( F_{(j_{s}'',j')}\)_{j_{s}''\in\mc{K}^{l}(\mc{M}^{l})} & = V^{l} \(F^{m_{p}}_{j'}\)_{\substack{m_{p} \in \mc{M}^l}} \labeleq{Fjl}
\end{align}
 for all $1\leq l \leq d$ where for a fixed $j'\in \RR^{d-l}$ we define	
\begin{align*}
F^{m_{p}}_{j'} &=  \sum_{\substack{m'\in\RR^{d-l}:\\(m_{p},m')\in\mc{M}^d }} F^{(m_{p},m')}  e^{-2\pi i \delta' j' \cdot m'}.
\end{align*}
In the case $l=1$, $V^1=W^1_1$, where $W^1_1$ is a $\#\mc{M}^1\times \#\mc{M}^1$ Vandermonde matrix $(W^1_1)_{st}=w_{t}^{s-1}$ with nodes $w_{t}=e^{-2\pi i \delta_{11} z_t}$,
$ z_t\in\mc{M}^1$. By the admissibility condition, the set $\{v_1 z_1 \mod q_1\mid z_1\in\mc{Z}^1_1\}$ contains $\#\mc{Z}^1_1$ distinct numbers and therefore the nodes $w_{t}$ are distinct, and $W^1_1$ is invertible for $\delta_{11}=v_1/q_1$.

For $l\geq 2$ and $p=0,\hdots, \#\mc{M}^{l-1}$ define the 1D Vandermonde matrices $\tilde{W}^{l}_{(q,p)}\in \mathbb{C}^{\#\mc{Z}^l_p
\times \#\mc{Z}^l_q}$ for $q<p$ and ${W}^{l}_{(p,p)}\in \mathbb{C}^{\#\mc{Z}^l_p
\times \#\mc{Z}^l_p}$ by
\begin{align*}
\begin{cases} (\tilde{W}^{l}_{(q,p)})_{st} = e^{2\pi i(s-1) z_{t} \delta_{ll}} & z_{t}\in \mc{Z}^{l}_{q}\\
({W}^{l}_{(p,p)})_{st} = e^{2\pi i(s-1) z_{t} \delta_{ll}} & z_{t}\in \mc{Z}^{l}_{p},
\end{cases}, \qquad 1\leq s\leq \#\mc{Z}^l_{p}. \quad 
\end{align*}
Define the matrices $\tilde{V}^{l-1}_{p}\in \mathbb{C}^{\#\mc{M}^{l-1}_p\times \#\mc{M}^{l-1}_q}$ and ${V}^{l-1}_{p}\in \mathbb{C}^{\#\mc{M}^{l-1}_p\times \#\mc{M}^{l-1}_p}$
\begin{align*}
\begin{cases}
(\tilde{V}^{l-1}_{p})_{sq} = e^{-2\pi i m_{q} \cdot \delta''j_{s}''}, & q<p\\
({V}^{l-1}_{p})_{s(q-p+1)} = e^{-2\pi i m_{q} \cdot \delta''j_{s}''} ,\quad  &q\geq p
\end{cases}, \quad  j''_{s}\in \mc{K}^{l-1}(\mc{M}^{l-1}_{p}), \qquad m_q\in \mc{M}^{l-1}_{p}.
\end{align*}
Then, we can form the linear systems from \refeq{Fjl} as
\begin{align}
\left( F_{(j'', j_{l},j')}\right)_{j''\in \mc{K}^{l-1}(\mc{M}^{l-1}_{1})} &=	V^{l-1}_{1}\( F^{m_q}_{(j_{l},j')}\)_{m_q\in \mc{M}_1^{l-1}}, \qquad \qquad j_l\in \mc{Q}^l_1 \labeleq{fmp1},\\
\left( F_{(j'', j_{l},j')}\right)_{j''\in \mc{K}^{l-1}(\mc{M}^{l-1}_{p})} - \tilde{V}^{l-1}_{p}\( F^{m_{q}}_{(j_{l},j')}\right)_{q<p}&=	V^{l-1}_{p}\( F^{m_q}_{(j_{l},j')}\)_{m_q\in \mc{M}_p^{l-1}}, \qquad \qquad j_l\in \mc{Q}^l_p \labeleq{fmpl}, \quad p\geq 2.
\end{align}
The vector $\( {F}^{m_{p}}_{(j_{l}, j')}\)_{j_{l} \in\mc{Q}^{l}_{i}, 1\leq i\leq p}$, created from solutions to \refeq{fmp1} and \refeq{fmpl} when $V^{l-1}_p$ are all invertible, solves the system
					\begin{align}
\( {F}^{m_{p}}_{(j_{l}, j')}\)_{j_{l} \in\mc{Q}^{l}_{i}, 1\leq i\leq p} &=W^{l}_{(p,p)}\left(F^{(m_{p},z_{l})}_{j'}\right)_{{z_{l}\in\mc{Z}^l_{p}}}		\labeleq{fmjlfin},
\end{align}
where the square matrices $W^{l}_{(p,p)}$ are invertible again by weak admissibility. 

To write $\refeq{fmp1} - \refeq{fmjlfin}$ in matrix form, we define $[V^{l-1}_p]$, $[\tilde{V}^{l-1}_p]$ to be block diagonal matrices with $\#\mc{Q}^l_p$ block entries all equal to $V^{l-1}_p$ and $\tilde{V}_p^{l-1}$, respectively, and 
\[ [\tilde{W}^l_p]=\begin{pmatrix}\tilde{W}^{l}_{(1,p)} & & &0 \\ 
& \tilde{W}^{l}_{(2,p)}&  &   \\
&&\ddots&\\
0 & & & \tilde{W}^{l}_{(p-1,p)}
\end{pmatrix}, \qquad [{W}^l_p]=\begin{pmatrix}{W}^{l}_{(1,1)} & & &0 \\ 
& {W}^{l}_{(2,2)}&  &   \\
&&\ddots&\\
0 & & & {W}^{l}_{(p-1,p-1)}
\end{pmatrix}.\]
The block diagonal matrix $[{W}^l_p]$ is invertible, and for appropriate permutation matrices\footnote{$U_p$ and $\tilde{U}_p$ correspond to the mappings $U_p( (v)_{ q\geq i, j_l\in \mc{Q}^l_i})_{i<p}=(( v )_{z_l\in\mc{Z}^l_q})_{q<p}$ and $ \tilde{U}_p ( (v)_{j_l\in \mc{Q}^l_p})_{ q<p}= ((v)_{ q<p})_{j_l\in \mc{Q}^l_p}$. } $\tilde{U}_p$, and $U_p$, the matrix defined as \[{X}_p=[{\tilde{V}}^{l-1}_p] \tilde{U}_p [\tilde{W}^{l}_{p}]([{W}^l_p])^{-1} U_p,\] satisfies
\begin{align*}
{X}_p \( \(F^{m_{q}}_{(j_l,j')}\)_{ q\geq i, j_l\in \mc{Q}^l_i}\)_{i<p}= \( \tilde{V}^{l-1}_{p} \( F^{m_q}_{(j_l,j')}\)_{q<p}\)_{j_l\in \mc{Q}^l_p}.\end{align*}
%
Therefore  $\refeq{fmp1}-\refeq{fmpl}$ can be written as
\begin{align}
  \begin{pmatrix}
\left(F_{(j'', j_{l},j')}\right)_{j''\in \mc{K}^{l-1}(\mc{M}^{l-1}_{1}), j_l\in \mc{Q}^l_1}\\
\left(F_{(j'', j_{l},j')}\right)_{j''\in \mc{K}^{l-1}(\mc{M}^{l-1}_{2}), j_l\in \mc{Q}^l_2}
\\\vdots\\
\left(F_{(j'', j_{l},j')}\right)_{j''\in \mc{K}^{l-1}(\mc{M}^{l-1}_{p}), j_l\in \mc{Q}^l_p}\\
\vdots\\
\end{pmatrix} = \begin{pmatrix}[{V}^{l-1}_1] & &\hdots & &0 \\ 
{X}_2 &  [{V}^{l-1}_2]& & & \\
&&\ddots&&\vdots \\
{X}_p & & & [{V}^{l-1}_p]&  \\
\vdots & & & & \ddots
\end{pmatrix}  \begin{pmatrix}
\(F^{m_{q}}_{(j_l,j')}\)_{q\geq 1, j_l\in \mc{Q}^l_1}\\
\(F^{m_{q}}_{(j_l,j')}\)_{q\geq 2, j_l\in \mc{Q}^l_2}\\
\vdots\\
\(F^{m_{q}}_{(j_l,j')}\)_{q\geq p, j_l\in \mc{Q}^l_p}\\
\vdots
\end{pmatrix}.\labeleq{FVFl}
\end{align}
The solutions to \refeq{FVFl} can then be rearranged to form the linear system,
\begin{align}
\( F^{m_{p}}_{(j_{l},j')}\)_{m_{p}\in\mc{M}^{l-1},j_{l}\in\cup_{i=1}^p\mc{Q}^{l}_i} & = W^{l} \(F^{(m_{p},z)}_{j'}\)_{\substack{(m_{p},z) \in \mc{M}^l}}\labeleq{wl}
\end{align}
where $W^l=[W^l_{(\#\mc{M}^{l-1}, \#\mc{M}^{l-1})}]$ is an invertible matrix.

Putting it all together, for an appropriate permutation matrix $U$, and the block lower triangular matrix in \refeq{FVFl} denoted by $[[{V}^{l-1}]]$, it follows that the matrix $V^l$ in \refeq{Vlinvertible} can be expressed as $V^l=[[{V}^{l-1}]]UW^l$, and it is invertible provided that ${V}^{l-1}_{p}$ are all invertible. The case $l=d$ is proved in the same way, with modifications in notation, implying that the terms $\{ F^{\lambda_r}(x)\}_{\lambda_r=m\in \mc{M}^d(x)}$ are determined uniquely.  Then, $f(x)$ can be uniquely constructed as \begin{align*}
f(x)&=\sum_{r=1}^kf(x)\chi_{\Omega_r}(x) = \sum_{r=1}^k  F^{\lambda_r}(\omega_r^{-1}(x))\chi_{\Omega_r}(x)\\
&
=\frac{1}{\text{vol}(\Lambda)}  \sum_{s=1}^k \sum_{\lambda^{*}\in \Lambda^{*}}  \langle f, {e}_{\lambda^{*}+a_{s}}\rangle 
\left(\sum_{r=1}^k  (V(x)^{-1})_{rs} (V(x))_{sr}\chi_{\Omega_r}(x)\) {e}_{\lambda^{*}+a_{s}}(x).\\
\end{align*}
Note that this reconstruction has the form $f(x) = \sum_{l\in L}c_l g_l(x)$ with $c_l=\frac{1}{|\Omega|}\langle f, e_l\rangle$ and $g_l(x)$ is defined as in Theorem \ref{thm:dualbasis}, however, weak admissibility of $\Omega$ does not guarantee that $g_l$ is a function in $L^2(\Omega)$.

\end{proof}

The proof can be summarized in the following pointwise reconstruction algorithm for functions $f\in L^2(\Omega)$, where $\Omega$ is a weakly $(v,q)$-admissible domain. For almost every $x'\in\Omega$, there is a unique $r\in \{1,\hdots k\}$ so that $x'\in \Omega_r$, so it suffices to state the algorithm for recovering $f(\omega_r(x)), x\in \Pi_\Lambda$.

\begin{alg}\label{alg:samp} 
Define $\mc{M}^{d}(x)=\mc{M}^{d}= \Lambda_x $. Construct the lattice shift set $\{a_s\}_{s=1}^k$ given by \refeq{shiftset} for a fixed diagonal matrix $\delta\in \mathbb{R}^{d\times d}$ satisfying the assumptions of Theorem \ref{thm:vander}. 
\begin{enumerate}[label=\textbf{Step \arabic*:}]
\setlength{\itemindent}{.5in}
\item \label{alg:step1}  For each $j \in \mc{K}^{d}(\mathcal{M}^{d})$,  define $F_j=F_{j}(x)$ by \refeq{Faj}.
\item If $d=1$, solve the 1D Vandermonde system \refeq{Fjl} for $l=1$, obtaining $\(F^{m_{1}}\)_{m_{1}\in \mathcal{M}^{1}}$. Then, skip to Step 4. If $d\geq 2$, solve \refeq{Fjl} with $l=1$ for all $j''\in \mc{K}^1(\mc{M}^1)$ and $j_1'=0\in \ZZ^{d-1}$, obtaining $\(F_{j'_1}^{m_{1}}\)_{m_{1}\in \mathcal{M}^{1}}$.

\item For $l=2, \hdots, d-1$ and $\mc{M}^{l-1}=\{m_q\}$
\begin{enumerate}
    \item For $p=1$ repeat the $l-1$ iteration (Step 2 for $l=2$ or Step 3 for $l>2$), substituting $j'_{l-1}$ with $(j_l, j'_l)$ where $j_l \in \mc{Q}^l_1$ and $j'_l=0 \in\RR^{d-l}$ to obtain $\(F_{(j_l, j'_l)}^{m_q}\)_{m_q\in \mc{M}^{l-1}_1}$. Then, determine $\(F_{j'_l}^{(m_1, z_l)}\)_{z_l\in \mc{Z}^l_{1}}$ by solving the 1D Vandermonde system \refeq{fmjlfin}.
    
    \item For $p=2, \hdots, \#\mc{M}^{l-1}$. Define 
    \[\tilde{F}_{(j'',j_l, j')}:=
    \left( F_{(j'', j_{l},j')}\right)_{j''\in \mc{K}^{l-1}(\mc{M}^{l-1}_{p})} - \tilde{V}^{l-1}_{p}\( F^{m_{q}}_{(j_{l},j')}\right)_{q<p}\]
    and $\tilde{\mc{M}}^d=\{m\in \mc{M}^d\mid (m_1,\hdots, m_l)\in \mc{M}^l_p\}$. Repeat Step 3 a) substituting $F_j$ and $\mc{M}^d$ with $\tilde{F}_j$ and $\tilde{\mc{M}}^d$,
    obtaining  $\(F_{j'_l}^{(m_p, z_l)}\)_{z_l\in \mc{Z}^l_{p}}$.
    
\end{enumerate}

\item For $l=d$, perform Step 3 omitting $j'_d$ to obtain \[f(\omega_r(x)) = {f}(\lambda_r(x)+x)\}_{\lambda_{r}\in\Lambda_x}.\]

\end{enumerate}

\end{alg}

Although Theorem \ref{thm:vander} guarantees that the matrix $V(x)$ in \refeq{V} is invertible, Algorithm \ref{alg:samp} does not directly find the inverse. Instead, the algorithm iteratively solves the system in a block-by-block fashion, and therefore only involves the inversion of 1D Vandermonde matrices. Equations \refeq{fmp1} and \refeq{fmpl} are only solved directly in the case $l=1$.

\begin{remark} \label{rem:kol}

The proof of Theorem \ref{thm:vander}  provides an explicit procedure for recovering functions that arise in, for example, \cite{Kolountzakis2013}, in which the existence of the set of vectors $\{a_{s}\}_{s=1}^{k}\subset \mathbb{R}^{d}$ is proved by showing the existence of $\Lambda$-periodic functions $\tilde{f}_{s}\in L^{2}(\Pi_{\Lambda}), 1\leq s\leq k$, such that
\begin{align}
{f}({\lambda_{r}(x)}+x) =  \sum_{s=1}^{k}e^{2\pi i a_{s}\cdot (x - \lambda_{r}(x))}\tilde{f}_{s}(x), \qquad  1\leq r\leq  k,
\end{align}
and showing that the determinant of the matrix for this system as a function of $x$ has finitely many zeros. Theorem \ref{thm:vander} provides one way to intuitively choose the lattice shifts for certain domains $\Omega$ by setting $\{a_{s}=\delta j_{s}\}_{s=1}^{k}$, $j_{s}\in\mc{K}^d(\mc{M}^d)$ for a suitable matrix $\delta\in\mathbb{R}^{d\times d}$.

\end{remark}

Algorithm \ref{alg:samp} provides a procedure for the pointwise (and adaptive) reconstruction of functions in $L^2({\Omega})$ that does not rely on a Riesz basis of exponentials. In fact, it may not be known {\it a priori} if $V$ in \refeq{Fr} satisfies the assumptions of Theorem \ref{thm:dualbasis}. The factorization of $V$ in the proof of Theorem \ref{thm:vander} provides a systematic procedure for choosing the lattice shifts $\{a_s\}$ and estimating the Riesz bounds (when they exist).  The algorithm finds a Riesz basis of exponentials for the following subset of domains satisfying  Definition~\ref{def:weakadmi}.

\begin{definition}\label{def:strongadmi}
Let $\Omega\subset\RR^d$ be a $k$-tile for a full lattice $\Lambda$. 
For $v\in\mathbb{R}^{d}$ and $q\in \ZZ_{+}^{d}$ we say that $\Omega$ is {\it strongly} $   (v,q)$-{\it admissible} if it is weakly $(v,q)$-admissible and for almost every $x\in \Pi_{\Lambda}$ the sets $\{v_{l} z_{l}\mod q_l \mid z_{l}\in \mc{Z}^{l}_p\}, 1\leq p \leq N_{l-1},  1\leq l\leq d$  are all sets of integers.
\end{definition}
For example, any bounded multi-rectangle in $\RR^d$ of the form $\Omega = \cup_{r=1}^k(\Pi_\Lambda+\lambda_r)$ with $\{\lambda_r\}\subset \Lambda$ is strongly admissible. The advantage of strong admissibility is the uniform boundedness of $\|V(x)\|$ which we will show next leads to a Riesz basis of exponentials.

\begin{corollary}Any strongly $(v,q)$-admissible $k$-tile for a full lattice $\Lambda$ admits a Riesz basis of exponentials $\{e_l\}_{l\in L}$ of the form given in Theorem \ref{thm:vander}. 
\label{cor:Riesz}
\end{corollary}
\begin{proof}
Let $\Omega\subset \RR^d$ be a strongly $(v,q)$-admissible $k$-tile for a full lattice $\Lambda$. Since the matrix $V(x)$ given in \refeq{V} is  piecewise constant on $\Pi_\Lambda$, there exists a normalized eigenvector of $V(x)$, denoted by $u(x)=(u_1(x),\hdots, u_k(x))$, that corresponds to an eigenvalue with squared magnitude $\sigma(x)$. It follows that $u(x)$ is also piecewise constant and each of its entries $u_r(x)$ are functions in $L^2(\Pi_\lambda)$.

Define the measurable function $f(x)\in L^2(\Omega)$ by $f(\omega_r(x))\equiv u_r(x)$ for $x\in\Pi_\Lambda$. Square-integrability follows immediately:   $|\Pi_\lambda|=\int_{\Pi_\lambda}\sum_{r=1}^k |u_r(x)|^2dx=\int_{\Pi_\lambda} \sum_{r=1}^k|f(\omega_r(x))|^2dx=\int_{\Omega} |f(x)|^2dx$. By taking the $\ell^2$-norm of  the vector $(F_{j_s}(x))_{s=1}^{k}=V(x) \(f(\omega_r(x))\)_{r=1}^{k}=V(x) u(x)$, 
\begin{align}	\sigma(x) =\|V(x)u(x) \|^2=\sum_{s=1}^{k}|F_{j_s}(x)|^2 \labeleq{sigmaf}. \end{align}
We will use the proof of Theorem \ref{thm:vander} to show that $\sigma(x)$ is uniformly bounded. Taking the $\ell_2$ norm of $\refeq{Fjl}$, for $l=1$ and $ j'\in \RR^{d-1}$, we obtain
\begin{align*}
\| V^{1}(x)^{-1}\|^{-2}\sum_{\substack{m_{1} \in \mc{M}^1}}| F^{m_1}_{j'}(x)|^2 &\leq 
\sum_{j_{1}\in\mc{K}^1({\mc{M}^{1}})} | F_{(j_{1},j')}(x)|^2
\leq \| V^{1}(x) \|^{2}\sum_{\substack{m_{1} \in \mc{M}^1}}| F^{m_{1}}_{j'}(x)|^2.
\end{align*}
Applying the inequality \refeq{wl} for $l=2,\hdots, d$,
\begin{align*}
\prod_{l=1}^{d}\| W^{l}(x)^{-1}\|^{-2} \sum_{\substack{m_{p} \in \mc{M}^{d}}}| F^{m_{p}}(x)|^2 &\leq 
\sum_{j\in\mc{K}^{d}(\mc{M}^{d})} | F_{j}(x)|^2 \leq \prod_{l=1}^{d}\| W^{l}(x)\|^{2} \sum_{\substack{m_{p} \in \mc{M}^{d}}}| F^{m_{p}}(x)|^2.
		\end{align*}
By \refeq{sigmaf} this implies that $\alpha \leq 	\sigma(x)  \leq \beta $,
where  \[\alpha=\inf_{x\in\Pi_\Lambda}\(\prod_{l=1}^{d}\| W^{l}(x)^{-1}\|^{-2}\), \qquad  \beta=\sup_{x\in\Pi_\Lambda}\(\prod_{l=1}^{d}\| W^{l}(x)\|^{2}\).\] 

The block matrices $W^l(x)$ in \refeq{wl} all have uniformly bounded norm because the square Vandermonde matrices $W^{l}_{(p,p)}(x)$ have nodes that form a subset of the $q_l$ roots of unity by the strong admissibility condition. Therefore $\alpha>0$, and $\beta<\infty$, and Theorem~\ref{thm:dualbasis} can then be applied to complete the proof.

\end{proof}

The Riesz bounds $A$ and $B$ give a sense of the stability of the reconstruction (\ref{def-rb1}). For the choice of lattice shifts $\{a_s\}$ given \refeq{shiftset}, these constants depend on the conditioning of the matrices $W^{l}(x)$. There are infinitely many feasible choices for $\delta\in \mathbb{R}^{d\times d}$ and $\eta\in \Lambda^*$ that can produce a Riesz basis of exponentials of the form \refeq{rb}, however, for $|\delta_{ll}|\ll1$ the condition numbers of these matrices grow large. It is, in general, a hard problem to estimate the condition number of a Vandermonde matrix \cite{Kunis2020,Aubel2017a,Gautschi1990}, and although the question of determining the optimal choice of $\delta$ is an important one, we consider this direction out of the scope of the present work. However, in special cases, we can show that the factorization of $V(x)$ produced by Algorithm \ref{alg:samp} can be used to derive optimally conditioned matrices. 

\begin{definition}\label{def-perfadmiss}
We say that a $k$-tile for a full lattice is {\it perfectly admissible} if there exists a vector $v$ such that it is strongly $(v,q^*)$-admissible, where $(q^*)_l=\#\mc{Z}^l_1=\hdots=\#\mc{Z}^l_{N_{l-1}}$. 
\end{definition}
 For example, any bounded multi-tiling for the full lattice $\Lambda=M\ZZ^d$ of the form $\Omega = \cup_{z\in \mathcal{M}^d} (\Pi_\Lambda+Mz)$ with $\mathcal{M}^d=\{Mz\mid z=(z_1, \hdots, z_d)^T\mid \in \ZZ^d, \mid |z_i|\leq K, 1\leq i \leq d\}$ for any $K\geq 1$ is perfectly admissible. We show next that perfect admissibility guarantees the existence of an orthogonal basis of exponentials.

\begin{corollary} Any perfectly admissible $k$-tile for a full lattice $\Lambda$ admits an orthogonal basis of exponentials $\{e_l\}_{l\in L}$ of the form \refeq{rb}.

\end{corollary}

\begin{proof} 
Define the lattice shifts $\{a_s\}_{s=1}^k$ by \refeq{shiftset} for the perfectly admissible domain $\Omega$ with $(v,q^*)$ as in Definition \ref{def-perfadmiss}. Then for $1< l\leq d$ and $1\leq p \leq \#\mc{M}^{l-1}$ the sets $\{v_l z\mid z\in \mc{Z}^l_p\}$ form a complete residue set \text{modulo} $\#\mc{Z}^l_p$, and it follows that each set $\mc{M}^l$ can be ordered so that $W^l_1=W^l_2=\hdots=W^l_p$ and the matrix $V$ in \refeq{V} has the form $V=W_1^1\otimes W^2_1\otimes\hdots \otimes W_1^d$. For the choice of $\delta_{ll} = v_l/\#\mc{Z}^l_1$, the nodes of $W^l_1$ form the $\#\mc{Z}^l_1$ roots of unity. Since the singular values of $V$ are products of the singular values of each factor, it holds that $\kappa(V(x))=\prod_l \kappa(W^l(x))=1$, where $\kappa(V(x))=\frac{\sigma_{\max}(V(x))}{\sigma_{\min}(V(x))}$ is the condition number of the matrix $V(x)$ for $x\in \Pi_\Lambda$, defined as the ratio of the maximum singular value $\sigma_{max}$ and minimum singular value $\sigma_{min}$.

Since $V(x)^*V(x)=kI \iff \kappa(V(x))=1$, Remark \ref{rem:dual} implies that the family $\{e_{l}\}_{l\in L}$ in \refeq{bio-basis} forms an orthogonal basis of exponentials for $L^2(\Omega)$.

\end{proof}

The following proves a partial converse result in one dimension.
\begin{corollary} Let $\Omega$ be $k$-tile for $\Lambda=\ZZ$. Suppose that there exists an orthogonal basis of exponentials for $L^2(\Omega)$ of the form $\{e_{\lambda^*+s \delta}\}_{\lambda^*\in \Lambda^*, 1\leq s \leq k}$ for some $\delta\in \RR$. If, in addition, there exists a number $Q\in \ZZ$ so that  $Q=\text{gcd}(\Lambda_x =\{z_1,\hdots, z_k\})$ for almost every $x\in \Pi_\Lambda$, then $\Omega$ is perfectly admissible.

\end{corollary}
\begin{proof}
For any set ${z_1,\hdots, z_k}\in\ZZ$, it is known that there exists a $\tau \in \mathbb{R}$ so that the Vandermonde matrix $V$ with entries $(V)_{st} =e^{\frac{2\pi i \tau s z_t}{k}}$ is perfectly conditioned, that is, $\kappa(V)=1$, if and only if $ \{ \frac{z_r}{Q} \text{ mod } k\}_{r=1}^{k}$ is a complete residue system, where $Q=\text{gcd}\{ z_r \}_{r=1}^{k}$ \cite{Bermant2007}. The choice $\tau =\pm\frac{1}{Q} + n k$ for an integer $n\in \ZZ$ produces a perfectly conditioned matrix.
For the choice $v=\tau/k$, and letting $q^*=k$,  $\Omega$ satisfies the conditions of perfect admissibility.

\end{proof}

\subsection*{Acknowledgements} The authors are grateful for discussions with David Walnut, Karamatou Yacoubou Djima, and Azita Mayeli. The authors also thank the anonymous reviewers for their helpful comments. The work of C.~Frederick is partially supported by the National Science Foundation grant DMS-1720306. K.~A.~Okoudjou was partially supported by  the National Science Foundation under Grant No.~DMS-1814253, and an MLK  visiting professorship at MIT.


\bibliographystyle{siam}
\bibliography{sampling.bib}
\newpage 
\begin{appendices}

\begin{lstlisting}[caption = {MATLAB code for finding $\mathcal{K}^d(\mc{M}^d)$},style=Matlab-editor,
basicstyle=\ttfamily,
escapechar=]
function K=K_l(Ml)
[n_Ml,l]=size(Ml);
%  If $l=1$, return $\{0, \hdots, \#\mathcal{M}^1-1\}$, otherwise, append the sets $K_i^l\times Q_i^l$
if l==1
    K = (0:(n_Ml-1))';
    return
else
    [Ml_old, n_Zl]=Ml_sort(Ml,l-1);
    [n_Ml, ~]=size(Ml_old);
    K=[];
    for i=1:n_Ml
        if i==1
            Qi_l = (0:(n_Zl(i)-1))';
        else
            Qi_l = (n_Zl(i-1):(n_Zl(i)-1))';
        end

        Ml_i=Ml_old(i:end,:);
        Kl_i=K_l(Ml_i);
        for ii=1:length(Kl_i(:,1))
            for jj=1:length(Qi_l)
                K=[K; Kl_i(ii,:) Qi_l(jj)];
            end
            
        end
    end
end

%  Frequencies in tiling lattice $\Lambda$
function [Ml n_Zl]=Ml_sort(Md,l)
if l==0
    Ml=Md;
    n_Zl=0;
else
    [Ml ia, ic] = unique(Md(:,1:l),'rows');
    n_Zl = histc(ic,unique(ic));
    [n_Zl, idx] = sort(n_Zl);
    Ml=Ml(idx,:);
end


\end{lstlisting}

\end{appendices}

\end{document}